\documentclass[reqno,11pt]{amsart}
\usepackage{amsthm,amsfonts,amssymb,euscript,mathrsfs,graphics,color,amsmath,amssymb,latexsym,marginnote}
\usepackage{soul}
\usepackage{graphicx}
\usepackage{amsthm,amsfonts,amssymb,euscript,mathrsfs,graphics,color,amsmath,amssymb,latexsym,marginnote}

\usepackage{hyperref}

\usepackage[margin=1.0in]{geometry}

\numberwithin{equation}{section}
 
\def\bq{\begin{equation}}
\def\eq{\end{equation}}

\def\ga{\gamma}

\def\ep{\epsilon}

\def\pa{\partial}

\def\sgn{\mathrm{sgn}}

\def\R{{\mathbb R}}

\def\sA{\mathscr A}

\def\BB{{\mathcal B}}
\def\KK{{\mathcal K}}
\def\CC{{\mathcal C}}

\def\HH{{\mathcal H}}
\def\LL{{\mathcal L}}
\def\GG{{\mathcal G}}

\def\WW{{\mathcal W}}

\def\SS{{\mathcal S}}

\def\SS{{\mathcal S}}

\def\KK{{\mathcal K}}

\def\RR{{\mathcal R}}

\def\VV{{\mathcal V}}

\def\Z{{\mathbb Z}}

\def\R{\mathbb{R}}

\definecolor{bluegreen}{rgb}{0.0, 0.3, 0.9}

\newtheorem{theorem}{Theorem}[section]
\newtheorem{lemma}[theorem]{Lemma}

\newtheorem{corollary}[theorem]{Corollary}

\newtheorem{remark}[theorem]{Remark}

\title{Amplitude bounds of steady rotational water waves}
\author{Susanna V.~Haziot}
\address{Department of Mathematics, Brown University, Providence, RI 02912, USA}
\email{susanna\_haziot@brown.edu}

\author{Walter A. Strauss}
\address{Department of Mathematics, Brown University, Providence, RI 02912, USA}
\email{walter\_strauss@brown.edu}
\begin{document}
\maketitle

\begin{abstract}
{We consider classical steady water waves with a free surface, a flat bottom and  constant 
vorticity $\ga$. .  
In the adverse case $\ga>0$ we prove that there is an absolute upper bound on the amplitude, independent of the physical constants, provided that $\ga$ is sufficiently small.  
In any favorable case $\ga\le0$ we present a new proof of such an absolute bound on the amplitude and prove that the amplitude tends to zero as $\ga$ tends to $-\infty$.  
}  
\end{abstract}


\section{Introduction}
\subsection{Main results} 

Traveling water waves are solutions of permanent form to the incompressible Euler equations which propagate at constant speed $c>0$ in a given direction, under the restoring force of gravity. They can be studied in a frame of reference that moves at the same speed $c$ as the wave, in which they will appear to be stationary, or \textit{steady}. The surface $\mathcal{S}$ of the wave is itself an unknown making it a free-boundary problem. The  waves, assumed to be constant in one horizontal direction, propagate through a fluid region $\Omega$ in the $(X,Y)$-plane that is bounded below by a flat bottom $\BB=\{Y=0\}$ and above by the free surface $\SS$. 
In this paper we consider waves that are $2\pi$-periodic in the remaining horizontal direction and have a constant vorticity $\ga$.  We assume that the parameterized surface $\SS=\{(\xi(s),\eta(s)):s\in\R\}$ lies strictly above $\BB$.  
Our goal is to find upper bounds on the amplitudes $\sA$ of the waves.  

Let $\Omega$ be the fluid domain in the moving frame.  The problem can be expressed in terms of the stream function $\Psi(X,Y)$, a function whose level curves coincide with the particle trajectories of the flow and which satisfies the Poisson equation together with suitable conditions on the boundary, as follows. 
\begin{equation}\label{stream}
\begin{aligned}
\Delta \Psi&=-\gamma&\quad&\text{in }\Omega,\\
\Psi&=0&\quad&\text{on }\mathcal{S},\\
\Psi&=-m&\quad&\text{on } \BB,\\
|\nabla\Psi|^2+2g\eta &=Q&\quad&\text{on }\mathcal{S}.
\end{aligned}
\end{equation}
Here we denote by $g>0$ the gravitational constant, $m\in\R$ the mass flux which we will choose to be negative in this paper,  
and $Q>0$ the Bernoulli constant (the energy) on the surface.  The \textit{vorticity}, or local swirling motion of the particles in the fluid, appears as the parameter $\gamma$.      
The velocity vector in the moving frame is $(u,v)=(\Psi_Y,-\Psi_X)$, in terms of which the vorticity is 
$\gamma=v_X-u_Y$ and  the mass flux $m$ is the integral of $u$ over any vertical slice.    
The physical velocity in the non-moving frame is $U=u+c$.  
It is obvious from \eqref{stream} that there is the trivial upper bound $\sA \le \frac Q{2g}$.  
Our goal is to find  much sharper bounds.  

This problem has been studied for centuries.  We begin here by stating a theorem of \cite {CSV1} 
that asserts the existence of a global curve of steady periodic water waves by means of bifurcation theory.   We state it somewhat informally.  

\begin{theorem}\cite{CSV1}
Fix the gravitational constant $g>0$, the conformal depth $d$ of the fluid and the vorticity 
	$\gamma\in\mathbb{R}$. There exists a continuous curve 
	$\mathcal{C} = \{(\Psi(s),m(s),Q(s)):\ 0\le s<\infty\}$    
of symmetric, $2\pi$-periodic solutions in 
	$C^{2,\alpha}(\mathbb{R})\times C^{2,\alpha}(\mathbb{R})$ of the system  \eqref{stream} for which \eqref{no_stag} and \eqref{nodal} are satisfied 
	with $m(s)<0$ and $Q(s)>0$.  
	This curve begins at the laminar flow and ``ends" with one of the following alternatives: 
	\begin{enumerate}
		\item either $\|\Psi_s\|_{ C^{2,\alpha}(\R)\times C^{2,\alpha}(\R)}  \to \infty$  as $s\to\infty$ 
		\item or $|m(s)|+Q(s) \to \infty$  as $s\to\infty$
		\item or $\min_{x\in\R}\{Q(s)-2g\eta_s(x) \} \to0$  as $s\to\infty$
		\item or there exists some $s^*\in(0,\infty)$ such that the wave profile self-intersects strictly above the trough.
	\end{enumerate}
\end{theorem} 


We present two theorems.  The first one is a proof of an upper bound on the amplitude in 
the {\it favorable} case $\ga\le 0$ .  The upper bound is similar but not identical to the 
one in \cite{CSV2} but the most important difference is that the proof is 
much simpler than the one in \cite{CSV2}.  
\begin{theorem}   \label{thm:favorable}
Consider any water wave that belongs to the bifurcation curve $\mathcal C$ 
in the favorable case $\gamma\le0$.  
Then the wave amplitude $\sA$ (the elevation difference between the crest and trough) 
satisfies  
	$\sA    <  \min\left\{2d , \frac1{|\ga|} \sqrt{12gd} \right \}$.  
Thus the amplitude  is uniformly bounded and it tends to zero as $\ga \to -\infty$.  
\end{theorem}  

Theorem \ref{thm:favorable} was proven in \cite{CSV2} with the different upper bound 
$2\pi/\beta(\frac\pi 2)$ by means of detailed quantitative estimates on the kernel.  
The function $\beta$ is the kernel of the Hilbert transform.  Note that in our theorem, 
for the shallow water case ($d$ small) the amplitude is small, which was not observed in \cite{CSV2}.
(We warn the reader that in \cite{CSV2} the vorticity $\gamma$ is denoted as $-\Upsilon$.)  

What we do here is to provide a very different, more qualitative, proof of a somewhat different upper bound.  Our proof is based on the direct use of the Hopf maximum principle, which we suspect may potentially turn out to be a more robust method applicable to a variety of other problems.  Our proof is given in Section \ref{sect:favorable}.  \bigskip 

Let us now turn to the main result of this paper.  It asserts an absolute upper bound on the amplitude in the {\it adverse} (unfavorable) case $\gamma>0$ provided $\gamma$ is sufficiently small.    

\begin{theorem}   \label{thm:adverse}  
	Consider a smooth water wave that belongs to the bifurcation curve $\mathcal C$ 
	in the adverse case $\gamma>0$ and assume that \textit{either} the slope $|\eta'|$ 
	or the convexity $|\eta''|$ of the wave is bounded.  
	Then the wave amplitude $\sA$ (the elevation difference between the crest and trough) 
	is uniformly bounded by a certain constant provided $\ga$ is sufficiently small.  
	The upper bound depends only on a certain explicit function of 
	the constants $g, Q, m$ and the conformal depth $d$.  
	
	As an example, the upper bound can be chosen to be the completely explicit number, ${12\pi}/ {\beta(\frac\pi 2)} < 104$, where $\beta$ is the kernel of the Hilbert transform, if we assume that each of the  quantities $\gamma, Q\gamma, |m|\gamma^2$, as well as either $N\ga^2$ or  $M\ga^4$, 
are less than certain explicit functions of $g$ and $d$.  These explicit functions are specified in either \eqref{N1}-\eqref{N2} or \eqref{ineqgamma1}-\eqref{ineqgamma2} of Section~\ref{sect:adverse}.    
Here $N:=\sup |\eta'|$ measures the maximum slope of $\SS$ and  $M:=|\eta''|$  measures its `convexity'.  This upper bound is certainly far from optimal.  
\end{theorem}

We remark that the smallness of $\gamma$ is not necessarily so severe, depending on the depth of the water. Indeed, $\gamma\leq O(d)+O(\sqrt{d})$, which is a modest condition if $d$ is large.
We provide two somewhat different proofs, one for which we control the slope $N$, 
and the other for which we control the convexity $M$.

\subsection{History} 
The study of steady water waves dates back to the 18th century with Stokes' work on the celebrated irrotational Stokes' wave. Recently there has been a flurry of work done on the subject, ranging from the construction of small and large amplitude rotational and irrotational waves to the study of more qualitative properties of the waves.  The reader is referred to \cite{survey} for a survey of recent results.  
 Of particular interest is the wave amplitude. The simple bound $\sA \le \frac Q{2g}$ is attained only for extreme waves but in general it is far from being sharp,   Besides, there is no guarantee that the Bernoulli constant $Q$ can be controlled.  Hence a more detailed analysis of the solutions is necessary for a better understanding of the wave amplitude.  

For rotational water waves the first rigorous amplitude bounds were derived in \cite{CSV2} for large amplitude, periodic, symmetric waves with constant \textit{favorable} vorticity. The waves lie on the solution curve $\mathcal C$ constructed in \cite{CSV1}.  It is proven that the waves cannot overhang, nor can they contain internal stagnation points. The bounds are very explicit but the proof involves some heavy machinery.  The authors of \cite{LWW} have recently generalized some of the amplitude bounds of \cite{CSV2} for general \textit{unidirectional} waves with favorable constant vorticity that need not lie on $\mathcal C$, nor be periodic or symmetric.   Waves are unidirectional if $\Psi_Y$ is either strictly negative or strictly positive throughout the fluid, which precludes stagnation points or overhanging wave profiles. However, the waves the authors studied do not necessarily lie on a bifurcation curve.  Their proof relies on a maximum principle argument and the bounds depend only on the value of the vorticity. 

In the adverse (unfavorable) setting, almost no theoretical studies have been carried out. Based on personal correspondence, Lokharu and Wheeler have considered unidirectional waves with very {\it large} adverse vorticity. Their bounds again depend solely on the value of the vorticity. 
On the other hand, the case of very small or just moderate adverse vorticity has been completely open.   
Moreover, in \cite{L}, Lokharu proved that in the unidirectional setting, $Q$ can be bounded by the vorticity.

Nevertheless, there are countless numerical studies of water waves, far too many to list here.  For instance, a comparative numerical study of water waves with adverse vorticity $\ga$ was carried out in \cite{KS} up to the creation of a stagnation point.  There it was shown that $\sA$ is increasing as a function of $\ga$ until it reaches stagnation at the crest, and it is also increasing as a function of the flux $|m|$.  

In this paper, we provide explicit absolute bounds on the wave amplitude for periodic symmetric waves in $\CC$ with very small adverse constant vorticity.

\subsection{Outline}
In Section~\ref{sect:reformulation}, we begin by reformulating the problem, first by using a conformal map to fix the free surface, and then by expressing the resulting problem essentially as a single equation depending only on the surface wave profile $\eta$.  Then we recall  the properties satisfied by the waves on the curve $\mathcal C$ of solutions constructed in \cite{CSV1}. 

In Section~\ref{sect:favorable}, we provide the elementary qualitative proof in the favorable vorticity setting, by relying entirely on the Hopf maximum principle.   

Finally, in Section~\ref{sect:adverse} we prove our main theorem that provides amplitude bounds for adverse vorticity $\ga$. We first provide new bounds for certain integral operators which are then used in a continuation argument to obtain upper bounds for the wave amplitude. The bounds depend on several physical parameters, one of which is the maximum slope $N$ of the wave profile.  We deduce the absolute bound stated in Theorem \ref{thm:adverse} provided we assume that $\ga$ is sufficiently small.  By a somewhat different calculation we obtain similar bounds in terms of the `curvature' $M$ of the wave profile instead of $N$.

\section{Reformulation}\label{sect:reformulation}
\subsection{Conformal transformation and non-local surface formulation}
Since the free surface $\mathcal{S}$ is itself an unknown to be solved for, we view 
the physical domain $\Omega$ as the image of the strip
\begin{equation}
\Omega^* :=\{(x,y)\in\R^2:-d< y < 0\}
\end{equation}
by  a conformal map $X+iY=\xi(x,y)+i\eta(x,y)$.   
The whole problem is thereby transferred to the strip.  
We denote the upper boundary of the conformal region by $\Gamma:=\{(x,y)\in\R^2:y=0\}$ and the bottom by $\{y=-d\}$ for a constant $d>0$ called the {\it conformal depth}. 
This is not the physical depth but numerics show that it typically differs only modestly from the physical depth.  
Taking our problem to have period $2\pi$, each period is the image of a rectangular region of length $2\pi$ as in  Figure~\ref{fig:conformal}. The details of the transformation were carried out in \cite{CSV1} and \cite{CSV2}. 

\begin{figure}
	\centering
	\includegraphics[scale=0.7]{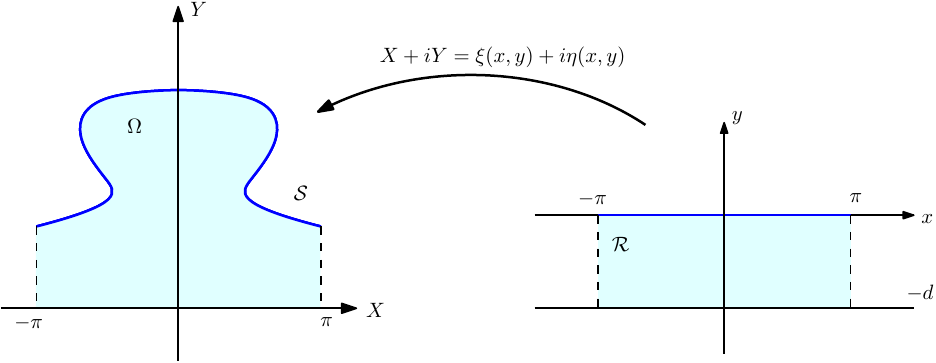}
	\caption{The conformal map}
	\label{fig:conformal}
\end{figure}

 Following \cite{CSV1}, 
we can reformulate the problem as a scalar integro-differential equation for the function $\eta(x,y)$ restricted to the top  $\Gamma$ of the rectangle. This nonlocal formulation involves the Hilbert transform of a periodic strip of depth $d$ defined as
\begin{equation}       \label{Hilbert}
\mathcal{H}f(x)=\sum_{n=1}^{\infty}  \{a_n\coth(nd)\sin(nx) -  b_n\coth(nd)\cos(nx)\},   
\end{equation} 
for any $2\pi$-periodic function $f(x) = \sum_{n=1}^{\infty} \{a_n\cos(nx)+b_n\sin(nx)\}$
with zero average that belongs to $L^2$.  

We denote the transformed stream function by $\psi(x,y):=\Psi(\xi(x,y),\eta(x,y))$. 
Because $\Delta \Psi = -\ga > 0$ in $\Omega$, we have $\Delta\psi>0$ in the strip $\Omega^*$.  
In terms of 
$\mathcal H$, we follow \cite{CSV1} to obtain 
\begin{equation}\label{psi_y}
\psi_y(x,0)=\frac{m}{d} + \frac{\gamma}{2d}[\eta^2]+\gamma\mathcal{H}(\eta\eta') - \gamma\eta\mathcal{H}(\eta')-\gamma\eta.
\end{equation}
Here and elsewhere the square bracket $[\,\,\cdot\,\,]$ denotes the average over a period of length $2\pi$.   
It was shown in \cite{CSV1} that the entire problem can be reformulated in terms of the dynamic boundary condition expressed in terms of $\eta$ and the periodic Hilbert transform.  
The dynamic boundary condition can then easily be expressed as 
\begin{equation}   \label{dynamic_hilbert}
\begin{split}
\bigg(\frac{m}{d} + \frac{\gamma}{2d}[\eta^2]+\gamma\mathcal{H}(\eta\eta') - \gamma\eta\big(1+\mathcal{H}(\eta')\big)\bigg)^2 =(Q-2g\eta)\Big((\eta')^2+(1+\mathcal{H}(\eta'))^2\Big).
\end{split}
\end{equation}
By the Riemann-Hilbert type methods in Theorem 1 of \cite{CSV1}, this  
equation \eqref{dynamic_hilbert} can be shown to be equivalent to the following equation 
\begin{subequations}
	\begin{equation}\label{babenko1}
	\begin{split}
	\mathcal{H}\Big((Q-2g\eta-\gamma^2\eta^2)\eta'\Big)+&\Big(Q-2g\eta-\gamma^2\eta^2\Big)\mathcal{H}(\eta')+2\eta\gamma^2\mathcal{H}(\eta\eta')\\&=\eta\bigg(-\frac{2m\gamma}{d}-\frac{\gamma^2}{d}[\eta^2]+2g\bigg)+\gamma^2\eta^2+2\gamma m-2gd-2g\big[\eta\mathcal{H}(\eta')\big], 
	\end{split}
	\end{equation}
where the quartic terms have disappeared, 
coupled to the scalar condition
	\begin{equation}   \label{equal_averages}
	\begin{split}
	\bigg[\bigg(\frac{m}{d} + \frac{\gamma}{2d}[\eta^2]+\gamma\mathcal{H}(\eta\eta') - \gamma\eta\big(1+\mathcal{H}(\eta')\big)\bigg)^2 \bigg]=\bigg[(Q-2g\eta)\Big((\eta')^2+(1+\mathcal{H}(\eta'))^2\Big) \bigg].
	\end{split}
	\end{equation}
\end{subequations}
These are the same as the equations (2.3) in \cite{CSV2}, where $\mathcal H$ was denoted by $\mathcal C_d$, the vorticity $\gamma$ was denoted by $-\Upsilon$, $d$ was denoted by $h$, and $\eta$ was denoted by $v$.  
We refer the reader to \cite{CSV1} for details of the derivation.

\subsection{Properties of the waves in $\mathcal{C}$}\label{sect:assumptions}
The waves along the curve $\mathcal{C}$ constructed in \cite{CSV1} enjoy all of the following properties which we now list. Since from this point on, the formulation of the problem will only be a scalar one-dimensional equation restricted to the surface $\Gamma$ of the conformal domain, we express these properties in terms of the function 
\begin{equation}\label{eta}
\eta_0(x):=\eta(x,0).
\end{equation}	 
We have the regularity $\eta_0\in C^{2,\alpha}(\Gamma)$ for some $\alpha\in(0,1)$.  
The periodicity and symmetry of the waves is captured by $\eta_0\text{ being }2\pi\text{-periodic and even in }x$. Furthermore, the average is $[\eta_0]=d$. In order for solutions to the problem in conformal variables to yield solutions in the physical plane,  $\text{the mapping }x\mapsto(\xi_0(x),\eta_0(x))\text{ must be injective on }\mathbb{R}$. Finally, the property
\begin{equation}\label{no_stag}
\inf_{\Gamma}(Q-2g\eta_0)^2|\nabla\eta|^2>0,
\end{equation}
ensures that there are no stagnation points on the surface and that $\eta$ is indeed the imaginary part of a conformal map. The waves in $\mathcal C$ have profiles that are strictly monotone in the vertical coordinate between crest ($x=0$) and trough ($x=\pi$).   Hence they satisfy the condition
\begin{equation}\label{nodal}
\partial_x\eta_0<0\text{ on }\Gamma\cap\{x\in(0,\pi)\}\qquad\text{and}\qquad\partial_x\eta_0>0\text{ on }\Gamma\cap\{x\in(-\pi,0)\}.
\end{equation}

Instead of working directly with $\eta_0$, it is convenient to introduce the function 
\begin{equation}   \label{def:f} 
f(x):=\frac{Q}{2g}-\eta_0(x). 
\end{equation} 
Thus $\sA = \eta_0(0)-\eta_0(\pi) = f(\pi)-f(0)$.  
As a direct consequence of the properties of $\eta$, we have the regularity
$f\in C^{2,\alpha}(\Gamma)$,
along with the periodicity and symmetry conditions that $f$  is $2\pi$-periodic even in $x$.  
Clearly $[f]=\frac{Q}{2g}-d$ and \eqref{no_stag} implies that
\begin{equation}
\begin{split}
f(x)>0\quad\text{ for all }x\in\R.
\end{split}
\end{equation}
The monotonicity of the wave profile is captured by the properties 
\begin{equation}
\begin{split}
f'(x)>0\quad&\text{ for all }x\in(0,\pi),\\ f\text{ is maximized at }&\pi\text{ and minimized at }0.
\end{split}
\end{equation}
In order to 
facilitate our presentation, as in \cite{CSV2} we introduce the constant $b$ 
and the operator $\KK$ defined as 
\begin{equation}\label{constants}
\begin{split}
b&:=\frac{Q}{2g} + \frac{\gamma m}{g}-d-[f\mathcal{H}f']-\frac{\gamma Q}{2g^2}\bigg(-\frac{m\gamma}{dg}+\frac{\gamma^2 Q^2}{8dg^3}-\frac{\gamma^2[f^2]}{2dg} \bigg) +\frac{\gamma^2Q^2}{8g^3}.
\end{split}
\end{equation}
\begin{equation}  \label{def:K}
\KK f := f^2\HH f'  + \HH(f^2f')  -2f \HH(ff').
\end{equation}
Expressing \eqref{babenko1} in terms of $f$ given by \eqref{def:f}, 
we obtain the following reformulation:  
\begin{subequations}\label{full_f}
	\begin{equation}\label{f}
	f   + \bigg(-\frac{m\gamma}{dg}+\frac{\gamma^2 Q^2}{8dg^3}-\frac{\gamma^2[f^2]}{2dg} \bigg)f   - \frac{\gamma^2}{2g} f^2  
	=  f\mathcal{H}f'+\mathcal{H}(ff') - \frac{\gamma^2}{2g}\KK f  + b.
	\end{equation} 
This is the same as equation (4.9) in \cite{CSV2}.  
Because our current interest is the amplitude $\sA = f(\pi)-f(0)$ of the wave, 
we evaluate the preceding the equation from the crest to the trough, that is, from $0$ to $\pi$, 
to obtain   
	\begin{equation}\label{f_diff}
	\left\{ 1-\frac{m\gamma}{dg}+\frac{\gamma^2 Q^2}{8dg^3}-\frac{\gamma^2[f^2]}{2dg} \right\}  f\bigg|_{0}^\pi
	-\frac{\gamma^2}{2g}f^2\bigg|_{0}^\pi=\{f\mathcal{H}f'+\mathcal{H}(ff') \}\bigg|_{0}^\pi-\frac{\gamma^2}{2g}  \KK f \bigg|_{0}^\pi.
	\end{equation}
\end{subequations} 
We will work directly with this equation.

\section{Favorable vorticity}             \label{sect:favorable}

In order to bound the amplitude in the favorable vorticity case, we begin by following the procedure of \cite{CSV2}. The key is to obtain upper and lower bounds for \eqref{f_diff}.  
In order to find an upper bound for the left hand side, we add the term  
$$
\frac{\gamma^2}{2g}(f(\pi)-f(0))\bigg(\Big(f\mathcal{H}f'-\mathcal{H}(ff') \Big)(\pi)+\Big(f\mathcal{H}f'-\mathcal{H}(ff') \Big)(0)\bigg)$$ 
to both sides of \eqref{f_diff} to obtain the identity 
\begin{equation}\label{fav1}
\begin{split}
&\bigg(1-\frac{m\gamma}{dg}+\frac{\gamma^2 Q^2}{8dg^3}-\frac{\gamma^2[f^2]}{2dg}\bigg)f\bigg|_{0}^\pi-\frac{\gamma^2}{2g}f^2\bigg|_{0}^\pi+\frac{\gamma^2}{2g}(f(\pi)-f(0))\bigg(\Big(f\mathcal{H}f'-\mathcal{H}(ff') \Big)(\pi)+\Big(f\mathcal{H}f'-\mathcal{H}(ff') \Big)(0)\bigg)\\&=f\mathcal{H}f'+\mathcal{H}(ff')\bigg|_0^\pi+\frac{\gamma^2}{2g}\bigg\{-\mathcal{K}f\bigg|_0^\pi+(f(\pi)-f(0))\bigg(\Big(f\mathcal{H}f'-\mathcal{H}(ff') \Big)(\pi)+\Big(f\mathcal{H}f'-\mathcal{H}(ff') \Big)(0)\bigg)\bigg\}.  
\end{split}
\end{equation}
For brevity we write this long equation as 
\[   \LL f = \RR f,  \]
where $\mathcal{K}$ was defined in \eqref{def:K}.    
The following upper bound for $\LL f$ appeared in \cite{CSV2}.  
\begin{lemma}\label{lem:Lf}
	Assume the favorable vorticity $\gamma\leq0$.   Then 
	\begin{equation}
	\mathcal{L}f\leq \sA := f\bigg|_0^\pi.
	\end{equation}
\end{lemma}
\begin{proof}  
We provide the proof for the sake of completeness.  
By definition, 
	\begin{equation}
	\begin{split}
	\mathcal{L}f&=\bigg(1-\frac{m\gamma}{dg}+\frac{\gamma^2 Q^2}{8dg^3}-\frac{\gamma^2[f^2]}{2dg}\bigg)f\bigg|_{0}^\pi-\frac{\gamma^2}{2g}f^2\bigg|_{0}^\pi\\&\qquad+\frac{\gamma^2}{2g}(f(\pi)-f(0))\bigg(\Big(f\mathcal{H}f'-\mathcal{H}(ff') \Big)(\pi)+\Big(f\mathcal{H}f'-\mathcal{H}(ff') \Big)(0)\bigg).
	\end{split}
	\end{equation}	
Evaluating \eqref{psi_y} at $f$ yields 
	\begin{equation}\label{psi_y_f}
		\psi_y(x,0)=\gamma\mathcal{H}(ff')-\gamma f\mathcal{H}(f')+\gamma f+\frac{m}{d}-\frac{\gamma Q^2}{8dg^2}+\frac{\gamma[f^2]}{2d}.
	\end{equation}
	From \eqref{no_stag} and the dynamic boundary condition \eqref{dynamic_hilbert}, we know that $\psi_y(x,0)$ is always either strictly positive or strictly negative. For the favorable case, $\gamma\leq0$, we have $\psi_y(x,0)\leq 0$, just as in the discussion at the beginning of section 3 in \cite{CSV2}. 
Multiplying \eqref{psi_y_f} by $-\gamma/g$, we get
	\begin{equation}\label{2}
	0 \geq  -\frac{\gamma}{g}\psi_y(x,0)=\frac{-\gamma^2}{g}\mathcal{H}(ff')+\frac{\gamma^2} {g} f\mathcal{H}(f')-\frac{\gamma^2}{g} f-\frac{m\gamma}{dg}+\frac{\gamma^2 Q^2}{8dg^3}-\frac{\gamma^2[f^2]}{2gd}.
	\end{equation}
This inequality is exactly where the assumption of favorable vorticity is crucial!  
Evaluating \eqref{2} at both $x=\pi$ and $x=0$ and taking the average yields
	\begin{equation}
	\begin{split}
	\frac{\gamma^2}{2g}\bigg(\Big(f\mathcal{H}f'-\mathcal{H}(ff') \Big)(\pi)+\Big(f\mathcal{H}f'-\mathcal{H}(ff') \Big)(0)\bigg)-\frac{\gamma^2}{2g} (f(\pi)+f(0))-\frac{m\gamma}{dg}+\frac{\gamma^2 Q^2}{8dg^3}-\frac{\gamma^2[f^2]}{2gd}\leq0.
	\end{split}
	\end{equation}
Adding $1$ to both sides and then multiplying both sides by $\sA = f(\pi)-f(0)$ yields 
exactly $\mathcal{L}f\leq \sA$.   This concludes the proof of the lemma. 
\end{proof}

\begin{remark}
	Lemma~\ref{lem:Lf} clearly relies on the sign of the vorticity. Indeed, this bound would not be valid in the adverse case, which leads to the need for a new approach to treat that situation, which we will do in Section\ref{sect:adverse}.
\end{remark}
Because we now have a simple upper bound of $\mathcal{L}f$, a suitable lower bound  for $\mathcal{R}f$ would put us in good shape for obtaining bounds on the amplitude. The following lemma asserts a lower bound on $\mathcal{R}f$.  
We provide it by a proof that is completely different from the one carried out in \cite{CSV2}, in that it relies solely on the maximum principle.  
We recall the notation that $\sA=f(\pi)-f(0)$ and that   $\mathcal{R}f$ is the right hand side of \eqref{fav1}.

\begin{lemma}\label{lem:Rf}
	For any $2\pi$-periodic even function $f$ that is strictly increasing on $(0,\pi)$, we have 
	\begin{equation}
	\mathcal{R}f >   \frac1{2d}\sA^2  +  \frac{\gamma^2}{12gd}\sA^3. 
	\end{equation}   
\end{lemma}
Temporarily assuming the validity of this lemma, the main result of this section follows easily. 
\begin{proof} [Proof of Theorem \ref{thm:favorable}]  
Combining Lemmas~\ref{lem:Lf} and \ref{lem:Rf}, we obtain 
	\begin{equation}
	\frac1{2d} \sA^2  +    \frac{\gamma^2}{12gd} \sA ^3<\mathcal{R}f=\mathcal{L}f\leq \sA .     
	\end{equation}
This yields
	\begin{equation}
	\sA    <  \min\left\{2d , \frac1{|\ga|} \sqrt{12gd} \right \}
	\end{equation}
 and concludes the proof of the theorem.  
\end{proof}

\begin{proof}[ Proof of Lemma \ref{lem:Rf}]
We split $\RR f$ into two parts 
	\begin{equation}\label{Rf}
	\begin{split}
	\mathcal{R}f:=\mathcal{V}+\mathcal{W}.  
	\end{split}
	\end{equation}
The quadratic part is
	\begin{equation}\label{V}
		\mathcal{V}:=\bigg(f\mathcal{H}f'+\mathcal{H}(ff')\bigg)\bigg|_0^\pi
	\end{equation}
and the cubic part is
	\begin{equation}\label{W}
		\mathcal{W}:=\frac{\gamma^2}{2g}\bigg\{-\mathcal{K}f\bigg|_0^\pi+(f(\pi)-f(0))\bigg(\Big(f\mathcal{H}f'-\mathcal{H}(ff') \Big)(\pi)+\Big(f\mathcal{H}f'-\mathcal{H}(ff') \Big)(0)\bigg)\bigg\}.
	\end{equation}
	
We begin with an analysis of $\VV$.  Let us denote $\HH' = \HH\pa_x = \pa_x\HH$.  
In terms of the pair of functions 
\begin{equation}\label{B}
	\begin{split}
	B^{\pi}(x)&:=-\big(f(\pi)-f(x)\big)\big(\tfrac{3}{2}f(\pi)+\tfrac{1}{2}f(x) \big)\\
	B^{0}(x)&:=-\big(f(x)-f(0)\big)\big(\tfrac{3}{2}f(0)+\tfrac{1}{2}f(x) \big),
	\end{split}
\end{equation}
we claim that 
\bq   \label{VBB}
	\VV = \HH'B^\pi(\pi) + \HH'B^0(0).  
\eq
Indeed, this is true because both sides of \eqref{VBB} are equal to 
\bq
f(\pi)\HH'f(\pi)  +   \tfrac12 \HH'(f^2)(\pi)  -  f(0)\HH'f(0)  -  \tfrac12 \HH'(f^2)(0).  
\eq
Notice that $B^\pi(x)$ is negative except at $x=\pi$ where it vanishes. 
Similarly, $B^0(x)$  is negative except at $x=0$ where it vanishes.

Now we introduce the Dirichlet operator $\GG$, which is defined as follows.  
If $\phi$ is any $2\pi-$periodic function, then 
\bq 
 (\GG\phi)(x) = \pa_y u(x,0),   \eq 
where $u(x,y)$ is the unique harmonic function in $\Omega^*$ that satisfies $ u(x,-d)=0,\ u(x,0)=\phi(x)$.  
It enjoys the following three properties.  

(i)  If $x_M$ is a point where $\phi(x)$ is maximized, then $\GG\phi(x_M) = \pa_yu(x_M,0)>0$.  
This is a direct consequence of the Hopf maximum principle.  
This inequality is the key to our proof.  

(ii)  $\pa_x\GG\phi = \GG\pa_x\phi$.  That is, $(\GG\phi)' = \GG(\phi')$.  

(iii)  \bq 
\GG\phi = \frac1d[\phi]  +  \HH'(\phi).  \eq   
Indeed, this follows directly from the Fourier expansions of $\phi(x)$ and $u(x,y)$.  

From property (iii) and \eqref{VBB}, we have 
\begin{equation}
	\mathcal{V}=\mathcal{G}B^\pi(\pi)-\frac{1}{d}\big[B^\pi \big]+\mathcal{G}B^0(0)-\frac{1}{d}\big[B^0\big].
\end{equation}
Because $B^\pi$ has its maximum of $0$ at $x=\pi$ and $B^0$ has its maximum of $0$ at $x=0$,  property (i) implies that
\begin{equation}
	\mathcal{G}B^\pi(\pi)>0\qquad\text{ and }\qquad\mathcal{G}B^0(0)>0.
\end{equation} 
Hence 
\begin{equation}
	\mathcal{V}>-\frac{1}{d}\big([B^\pi]+[B^0] \big). 
\end{equation} 
We note that 
\bq
-B^\pi-B^0 = -\frac12f^2 - f(\pi)f + \frac32f^2(\pi)  + \frac12f^2 + f(0)f - \frac32f^2(0)  
		=  \sA \{-f +\frac32f(\pi) + \frac32 f(0) \}.     \eq
Hence 
\bq		
-(B^\pi+B^0) > \frac12\sA \{f(\pi)+3f(0)\}  > \frac12\sA^2 ,  \eq 
whence 
\bq    \label{lowerboundV}
\VV  >  \frac1{2d} \sA^2.  \eq  

\bigskip  Now we analyze $\WW$.  
We define the function
	\begin{equation}\label{lab2}
		S(x,y):=\frac{1}{6}\big(f(y)-f(x) \big)^2\big(3f(\pi-y)-2f(x)-f(y) \big), 
	\end{equation} 
as well as 
	\begin{equation}\label{Spi}
	\begin{split}
	&S^\pi(x):=S(x,\pi)=\frac{1}{6}\big(f(\pi)-f(x) \big)^2\big(3f(0)-2f(x)-f(\pi) \big)
	\end{split}
	\end{equation}
and
	\begin{equation}\label{S0}
	\begin{split}
	&S^0(x):=S(x,0)=\frac{1}{6}\big(f(0)-f(x) \big)^2\big(2f(x)+f(0)-3f(\pi) \big). 
	\end{split}
	\end{equation}
Notice that $S^\pi(x)$ is negative except at $x=\pi$ where it vanishes. 
Similarly, $S^0(x)$  is negative except at $x=0$ where it vanishes.  
Expanding both functions, we calculate $S^\pi(x)+S^0(x) = -\frac16\sA^3$.  
	
We claim that  
\bq    \label{WHS}
\frac{2g}{\gamma^2}\mathcal{W}=\HH'S^\pi(\pi) + \HH'S^0(0).   \eq
In order to prove the claim, an expansion of \eqref{W} into its eight individual terms yields 
\bq    \label{Wexpanded}
\frac{2g}{\ga^2} \WW  =  -\tfrac13\HH'f \bigg|_0^\pi  +\tfrac12 (f(\pi)+f(0)) \HH'(f^2) \bigg|_0^\pi   
- f(0)f(\pi) \HH'f \bigg|_0^\pi .  
\eq
On the other hand, expansions of $S^\pi$ and $S^0$ into their individual terms, followed by an application of the operator $\HH'$, yields 
\bq    \label{HSpi}
\HH'(S^\pi) = -\tfrac13\HH'(f^3)  +\tfrac12(f(\pi)+f(0)) \HH'(f^2) - f(0)f(\pi)\HH'f  \eq 
and 
\bq    \label{HSzero}  
\HH'(S^0) = \tfrac13\HH'(f^3)  -\tfrac12(f(\pi)+f(0)) \HH'(f^2) + f(0)f(\pi)\HH'f  \eq 
Evaluation of \eqref{HSpi} at $\pi$ and \eqref{HSzero} at $0$ and adding them yields exactly 
\eqref{Wexpanded}.  This proves the claim.  

By property (iii) of $\GG$, we can rewrite \eqref{WHS} as 
	\begin{equation}\label{Q4}
	\frac{2g}{\gamma^2}\mathcal{W}=\mathcal{G}S^\pi(\pi)-\frac{1}{d}\big[S^\pi\big]+\mathcal{G}S^0(0)-\frac{1}{d}\big[S^0\big]
	\end{equation}	
Because $S^\pi$ has its maximum of $0$ at $x=\pi$ and $S^0$ has its maximum of $0$ at $x=0$,  property (i) of $\GG$ implies that
	\begin{equation}
		\mathcal{G}S^\pi(\pi)>0\qquad\text{and}\qquad\mathcal{G}S^0(0)>0,
	\end{equation}
so that 
	\begin{equation}   \label{lowerboundW}
	\frac{2g}{\gamma^2}\mathcal{W}
	>  -\frac{1}{d}\Big[S^\pi + S^0\Big] 
	=  \frac 1{6d} \sA^3
	\end{equation}
Insertion of the lower bounds \eqref{lowerboundV} and \eqref{lowerboundW} into 
 $\RR f = \VV + \WW$ completes the proof of Lemma \ref{lem:Rf} and therefore of Theorem \ref{thm:favorable}.


\end{proof}

\section{ Adverse Vorticity}\label{sect:adverse}

This section is devoted to the adverse case.  

\subsection{The Periodic Hilbert Transform}

Instead of the Dirichlet-Neumann operator, in this section we will use the explicit integral representation of the Hilbert transform.  
The periodic Hilbert transform $\mathcal{H}$ was defined in \eqref{Hilbert}.  
\begin{lemma}
	$\HH$ can be expressed as a convolution, namely,  
	\begin{equation}\label{H}
	\mathcal{H}(F)(x)=\frac{1}{2\pi}\textup{p.v.}\int_{-\pi}^\pi\beta(x-s)F(s)ds,\qquad x\in\R
	\end{equation}
	for any smooth $2\pi$-periodic function $F:\R\to\R$ with mean zero.  The kernel $\beta:\R\setminus2\pi\Z\to\R$ is given by the series 
	\begin{equation}\label{beta}
	\beta(s)=-\frac{s}{d}+\frac{\pi}{d}\sum_{n=-\infty}^{\infty}\bigg\{\operatorname{coth}\Big(\frac{\pi}{2d}(s-2\pi n)\Big)+\sgn(n) \bigg\}.
	\end{equation}
	It is $2\pi$-periodic and odd.  
	It has an $s^{-1}$ singularity at $s=0$ and is smooth except at the integer multiples of $2\pi$.  
	It is strictly decreasing on $(0,2\pi)$, where it goes from $=\infty$ to $-\infty$.  .  
\end{lemma}
\begin{proof}
	The proof was given in \cite{CSV2}.  
\end{proof}

Lower and upper bounds of the kernel $\beta$ are given in the next two lemmas.  
\begin{lemma}    \label{lem:beta}
	Given $d>0$, $\beta(s)$ is a positive function of $s\in(0,\pi]$, strictly decreasing from $+\infty$ to $0$. Furthermore,
	\begin{equation}
	\beta\left(\frac{\pi}2\right)\geq\frac{\pi-2}{\pi} > 0.363 
	\qquad\text{for all }d\in(0,\infty).
	\end{equation}
	Moreover, the function $\beta$ is odd, $2\pi$-periodic and strictly decreasing on $(0,2\pi)$.
\end{lemma} 
\begin{proof}
	The proof was given in \cite{CSV2}.  The optimal lower bound is unknown.  
\end{proof}

\begin{lemma}\label{lem:bounds_beta_der}
	Given $d>0$, let $\beta$ be defined by \eqref{beta}. Then for $s\in[-\pi,\pi]$ we have 
	\begin{equation}
	-\beta'(s)<\frac{1}{d}+\frac{2}{s^2}+\frac{1}{2}.
	\end{equation}
\end{lemma}
\begin{proof}
	From \eqref{beta}, we calculate
	\begin{equation}
	-\beta'(s)=\frac{1}{d}+\frac{\pi^2}{2d^2}\sum_{n=-\infty}^\infty\frac{1}{\sinh^2(\frac{\pi}{2d}(s-2\pi n))}.
	\end{equation}
	Using the the fact that 
	\begin{equation}
	|\sinh(x)|\geq|x|,
	\end{equation}	
	we get
	\begin{equation}\label{estimate}
	\begin{split}
	-\beta'(s)&<\frac{1}{d}+\frac{2}{s^2}+\frac{\pi^2}{2d^2}\sum_{n\geq1}\frac{4d^2}{\pi^4(s-2\pi n)^2}+\frac{\pi^2}{2d^2}\sum_{n=-\infty}^{-1}\frac{4d^2}{\pi^4(s+2\pi n)^2}\\
	&=\frac{1}{d}+\frac{2}{s^2}+\frac{1}{2},
	\end{split}
	\end{equation}
	for $d>0$.
\end{proof}
In particular, we see that the derivative 
\begin{equation}
-\beta'(s)\text{ is of type }O(1/s^2)\text{ as }s\to0.
\end{equation}


\subsection{Bounds on the quadratic and cubic terms}	
Once again our starting point is the identity \eqref{f_diff}.  
In order to obtain bounds on the wave amplitude we will exploit the integral formulation. 
From \eqref{H} we see that
\begin{equation}\label{HF}
\mathcal{H}(F')(x)=-\frac{1}{2\pi}\textup{p.v.}\int_{-\pi}^{\pi}\beta'(s)\big(F(x)-F(x-s) \big)ds,\qquad x\in\R,
\end{equation}
for any smooth $2\pi$-periodic function $F:\R\to\R$. This enables us to rewrite the quadratic terms on the left side of \eqref{f}. We get
\begin{subequations}\label{left}
	\begin{equation}
	f\mathcal{H}(f')(x) = \textup{p.v.}\int_{-\pi}^{\pi}\frac{-\beta'(s)}{2\pi} f(x)\big(f(x)-f(x-s) \big)ds
	\end{equation}
and
	\begin{equation}
	\mathcal{H}(ff')(x)=\textup{p.v.}\int_{-\pi}^\pi\frac{-\beta'(s)}{4\pi}\big(f^2(x)-f^2(x-s) \big)ds,
	\end{equation}
	as well as the cubic expression 
	\begin{equation}\label{K}
	\mathcal{K}f(x):=\big(f^2\mathcal{H}f'+\mathcal{H}(f^2f')-2f\mathcal{H}(ff') \big)(x)=\int_{-\pi}^\pi\frac{-\beta'(s)}{6\pi}\big(f(x)-f(x-s) \big)^3ds,
	\end{equation}
	for all $x\in\R$.
\end{subequations}
Subtracting, we also have 
\begin{equation}
\Big(f\mathcal{H}f'-\mathcal{H}(ff') \Big)(x)=\int_{-\pi}^\pi\frac{-\beta'(s)}{4\pi}\big(f(x)-f(x-s)\big)^2ds.
\end{equation}
The bounds in the next two lemmas will serve as building blocks.

\begin{lemma}  \label{lem:bounds_CSV}
	For any smooth, strictly positive $2\pi$-periodic even function $f$ which is strictly increasing on $(0,\pi)$, we have 
	\begin{equation}
	\{f\mathcal{H}f'+\mathcal{H}(ff') \}\bigg|_{0}^\pi\geq 
	\frac{\beta(\pi/2)}{2\pi} \sA^2, 
	\end{equation} 
	where $\beta$ is defined in \eqref{beta}.	
\end{lemma}
\begin{proof} 
This estimate on the quadratic terms occurred in the proof of Lemma 2 in \cite{CSV2}, 
where the left-hand side was denoted as $I+II$.  
Indeed, in that reference $I \ge \frac\beta \pi f(0)\sA$  and  $II \ge \frac\beta{2\pi} (f^2(\pi)-f^2(0))$.  
Hence $I+II \ge \frac\beta\pi \sA \{f(0)+\frac12(f(\pi)+f(0))\}$, which implies the estimate of the lemma.  
	
\end{proof}

\begin{lemma}\label{lem:bounds_K}
For any smooth positive $2\pi$-periodic even function $f$ that is strictly increasing in 
the interval $(0,\pi)$, we have
	\begin{equation}
	\mathcal{K} f\Big|_0^\pi  
	<\frac{2}{3}\bigg(\frac{1}{d}+\frac{1}{2}\bigg)  \sA^3  
	+   \frac{8 N }{3\pi} \sA^2 ,  
	\end{equation}
	where $\mathcal{K}$ is defined  in \eqref{K}.and $N=\sup_x|\eta'(x)|$ is the maximum slope.   
\end{lemma}

\begin{proof}
	From Lemma~\ref{lem:bounds_beta_der} we can approximate $\mathcal{K}$ to get
	\begin{equation}
	\mathcal{K}f\Big|_0^\pi <  \frac{1}{6\pi}
	\int_{-\pi}^\pi  \bigg(\frac{1}{ d}+\frac{1}{2}+\frac2{s^2}\bigg)\  \Big\{ (f(\pi)-f(\pi-s))^3  - (f(0)-f(-s))^3 \Big\} ds .
	\end{equation}  
The integrand is an even function of $s$.  
Although the bracketed expression 
	$\mathcal{I}(s)  := (f(\pi)-f(\pi-s))^3 + (f(s)-f(0))^3 $.  
 is clearly less than $2\sA^3$, it does not take into account the singularity at $s=0$. 
 So we begin by splitting the integral  as 
	\begin{equation}\label{K1}
	\begin{split}
	\mathcal{K}f\Big|_0^\pi &< \frac{1}{3\pi}
	\int_{0}^\pi  \bigg(\frac{1}{ d}+\frac{1}{2}\bigg)\  \Big\{ (f(\pi)-f(\pi-s))^3  - (f(0)-f(-s))^3 \Big\} ds\\
	&\qquad+\frac{2}{3\pi}
	\int_{0}^\pi  \frac1{s^2}\  \Big\{ (f(\pi)-f(\pi-s))^3  - (f(0)-f(-s))^3 \Big\} ds\\
	&=:\mathcal{J}_1+\mathcal{J}_2.
	\end{split}
	\end{equation}
From these considerations we very easily bound
	\begin{equation}
	\mathcal{J}_1\leq\frac{2}{3}\bigg(\frac{1}{d}+\frac{1}{2}\bigg)\sA^3.
	\end{equation}
In order to handle $\mathcal{J}_2$, 
we will make use of the behavior of  two of the three factors in the cubic expression \eqref{K1} at $s=0$. Indeed, since $f'(0)=f'(\pi)=0$, a Taylor expansion with remainder gives us 
	\begin{equation}\label{K2}
	f(s)-f(0)     \le Ns \quad \text{ for } s\in (0,\delta),
	\end{equation} 
	and similarly for $f(\pi)-f(\pi-s)$.  
We proceed to break the interval of integration in \eqref{K1} into $(0,\delta)\cup(\pi-\delta,\pi)$. 
Using the simple estimate \eqref{K2} on two of the factors in $\mathcal I(s)$, we get 
	$\mathcal I(s)\le 2(Ns)^2\sA$.  
	The $s^2$ singularity cancels and 
	\begin{equation}
	\int_0^\delta\frac{\mathcal{I}(s)}{s^2} \,ds   \le 2N^2\sA\delta.  
	\end{equation}  
The remaining integral is easily bounded as 
	\begin{equation}
	\int_\delta^\pi \frac{\mathcal{I}(s)}{s^2} \,ds \le  2\int_\delta^\pi \frac{ds}{s^2} \sA^3  
	= \frac2\delta \sA^3 . 
	\end{equation} 
	Combining, we thus obtain the estimate  
	
	\begin{equation}\label{K_bounds}
	\mathcal{K} f\Big|_0^\pi  <\frac{2}{3}\bigg(\frac{1}{d}+\frac{1}{2}+\frac{2}{\pi\delta}\bigg)\sA ^3  
	+  \frac{4}{3\pi}N^2\delta \sA . 
	\end{equation}
The parameter $\delta$ is available to be chosen.	
In order to minimize the right side of \eqref{K_bounds}, we choose $ \delta=\sA / N $, which yields the upper bound 
	\begin{equation}
	\mathcal{K} f\Big|_0^\pi  <\frac{2}{3}\bigg(\frac{1}{d}+\frac{1}{2}\bigg)\sA ^3  
	+  \frac{8}{3\pi} N\sA^2.  
	\end{equation} 
	This completes the proof of the lemma.  
\end{proof}

\subsection{Proof of the main result}\label{sect:bounds}
We are now in a position to prove Theorem \ref{thm:adverse}. We begin with the following theorem, in which we control the slope of the wave.

\begin{theorem}
	Given a solution that belongs to $\mathcal{C}$, let the conformal depth $d>0$ and the bound $N>0$ on the slope be fixed. 
Then the amplitude is bounded as follows.  
	\begin{equation} 
	\sA  \le  \frac12(E-D) - \frac12 \sqrt{(E-D)^2-4F}  \quad \text{ provided } \quad E>D \quad  \text{ and } \quad 4F<(E-D)^2,  
	\end{equation}
	where $E,F$ and $D$ are provided below in \eqref{bounds4}.
\end{theorem}


\begin{proof}
	For convenience we copy the key identity from \eqref{f_diff}, as follows:    
	\begin{equation}
	\left\{ 1-\frac{m\gamma}{dg}+\frac{\gamma^2 Q^2}{8dg^3}-\frac{\gamma^2[f^2]}{2dg} \right\}  f\bigg|_{0}^\pi
	-\frac{\gamma^2}{2g}f^2\bigg|_{0}^\pi
	=  \{f\HH f'+\HH (ff') \}\bigg|_{0}^\pi - \frac{\gamma^2}{2g} \KK f \bigg|_{0}^\pi.
	\end{equation}
Recalling  Lemmas~\ref{lem:bounds_CSV} and \ref{lem:bounds_K} 
and the notation $\sA =f(\pi)-f(0)$, we obtain  
	\begin{equation}
	\begin{split}
	\left(1-\frac{m\gamma}{dg}+\frac{\gamma^2 Q^2}{8dg^3}-\frac{\gamma^2[f^2]}{2dg} \right)\sA  
	&\geq\bigg(\frac{\gamma^2}{2g}+\frac{\beta(\pi/2)}{2\pi}\bigg)\sA ^2
	-\frac{\gamma^2}{3g}\bigg(\frac{1}{d}+\frac{1}{2}\bigg)\sA ^3 
	-  \frac{4\gamma^2}{3\pi g} N\sA^2.   
	\end{split} 
	\end{equation}
	Dividing by $\sA $ yields
	\begin{equation}\label{bounds3}
	\begin{split}
	1-\frac{m\gamma}{dg}+\frac{\gamma^2 Q^2}{8dg^3}-\frac{\gamma^2[f^2]}{2dg} 
	\geq\bigg(\frac{\gamma^2}{2g}+\frac{\beta(\pi/2)}{2\pi}\bigg)\sA 
	-\frac{\gamma^2}{3g}\bigg(\frac{1}{d}+\frac{1}{2}\bigg)\sA ^2
	-  \frac{4\gamma^2}{3\pi g} N\sA.   
	\end{split}
	\end{equation}
Dividing by the leading coefficient of $\sA$, we obtain the quadratic inequality
	\begin{equation}\label{quadratic}
	\sA ^2+(D-E)\sA +F\geq0,  
	\end{equation}
where 
	\begin{equation}\label{bounds4}
	\begin{split}
	D&:=    \frac{4N}{\pi(\frac1d+\frac12)},      \\    
	E&:       =\frac{3\beta(\pi/2)+3A\pi}{2A\pi(\frac{1}{d}+\frac{1}{2})}
	= \frac{3}{2(\frac{1}{d}+\frac{1}{2})} 
	+ \frac{3g\beta(\frac\pi 2)} {2\pi\gamma^2 (\frac{1}{d}+\frac{1}{2})} ,\\ 
	F&:=\frac{3g}{\gamma^2(\frac{1}{d}+\frac{1}{2})} 
	\bigg(1-\frac{m\gamma}{dg}+\frac{\gamma^2 Q^2}{8dg^3}-\frac{\gamma^2[f^2]}{2dg} \bigg)= \frac{3g}{\frac{1}{d}+\frac12} \left\{ \frac1{\gamma^2} - \frac{m}{\gamma gd} 
	+ \frac1{2dg}\left( \frac{Q^2}{4g^2}-[f^2]\right) \right\} .  
	\end{split}
	\end{equation}
Clearly, both $D$ and $E$ are positive. 
Note that both $E$ and $F$ have order $O(\gamma^{-2})$ as $\gamma\to 0$.  
Note also that $f\le \frac{Q}{2g}$, so that $ \frac{Q^2}{4g^2} > [f^2] >0$. 
{\it Here we are assuming the adverse case $\gamma>0$ and are using the fact that $m<0$, so that $F$ is positive as well. } 
	
Now we make use of the quadratic inequality \eqref{quadratic}.  
	Desiring the parabola in \eqref{quadratic} to cross the horizontal axis,  we require  
	
	\begin{equation}
	E>D \quad  \text{ and } \quad 4F<(E-D)^2. 
	\end{equation}
From \eqref{quadratic} it follows that 
	\begin{equation}   \label{A bound2}
	\sA  \le  \frac12(E-D) - \frac12 \sqrt{(E-D)^2-4F},   
	\end{equation}
where the right side is the smaller root of the quadratic.    Now the curve $\mathcal C$ of solutions includes the trivial wave that has zero amplitude.  The key fact is that $\mathcal C$ is a connected set.  Therefore all the solutions in it must indeed satisfy \eqref{A bound2}.  
\end{proof} 

In order to exhibit a more explicit bound, we now make some specific choices.   

\begin{corollary}
	Provided the bounds
	\begin{equation} \label{N1}
	\gamma<\frac{gd}{|m|}, \quad \text{ and }\quad \gamma^2<\frac{8dg^3}{Q^2}, 
	\end{equation} 
	and 
	\begin{equation} \label{N2} 
	N\gamma^2 < \frac{\beta(\frac\pi{2}) g}{8}   \quad \text{ and } 
	\gamma^2 < \frac{1}{\frac12 + \frac1d}  \frac{g\beta^2(\frac\pi{2})} {77\pi^2},
	\end{equation}
	hold, then the amplitude is bounded by
	\begin{equation}
	\mathscr A<\frac{12\pi}{\beta(\frac\pi2)}.
	\end{equation} 	
\end{corollary}
\begin{proof}
We will require an explicit estimate of the right side of \eqref{A bound2}.  
By the Taylor expansion of the square root with remainder, 
we have 
	\begin{equation}   \label{Taylor} 
	\sqrt{(E-D)^2-4F} = (E-D)+\frac12(E-D)^{-1}(-4F)  -  \frac18 \Xi^{-3/2} (-4F)^2  
	= (E-D) - \frac{2F}{E-D} - 2F^2 \Xi^{-3/2}, 
	\end{equation}
where $(E-D)^2-4F < \Xi < (E-D)^2$.  So from \eqref{A bound2} we have 
\bq    \label{A bound3} 
\sA  \le \frac{F}{E-D}  +  F^2 \Xi^{-\frac32}.  \eq 
Choose $\ga$ so small that 
	\begin{equation}  \label{NmQ ineq}
	N\ga^2 <  \frac{g\beta(\frac\pi{2})}{8},   \qquad |m|\ga < gd, \qquad  \ga^2 Q^2 < 8g^3d . 
	\end{equation} 
Then the quantities $D,E,F$ given above satisfy 
	\begin{equation}   \label{DEFineq}
	\left(\frac1d+\frac12\right) F < \frac{9g}{\ga^2}\ ,  \qquad  
	\left(\frac1d+\frac12\right) (E-D)  >  \frac{g\beta(\frac\pi{2})}{\pi\ga^2} .
	\end{equation}
So 
\[
	\Xi  >  (E-D)^2-4F  >  \frac1{\left(\frac1d+\frac12\right)^2} \frac{g^2\beta(\frac\pi{2})^2} {\pi^2\ga^4}  -  \frac1{\left(\frac1d+\frac12\right)} \left( \frac{36g}{\ga^2} \right)   
	> \frac12 \frac1{\left(\frac1d+\frac12\right)^2} \frac{g^2\beta(\frac\pi{2})^2} {\pi^2\ga^4}, 
	\] 
	the last inequality being true provided   
	\begin{equation}  \label{gamma ineq1} 
	\ga^2 < \frac1{77}\frac1{\left(\frac1d+\frac12\right)} \frac{g\beta(\frac\pi2)^2} {\pi^2} . 
	\end{equation}
	So the remainder in \eqref{Taylor} is bounded above:  
	\[
	F^2 \Xi^{-3/2}  <  \left( \frac{9g} {\frac1d+\frac12} \frac1{\ga^2} \right)^2 
	\left\{  \frac12 \frac1{\left(\frac1d+\frac12\right)^2} 
	\frac{g^2\beta(\frac\pi{2})^2} {\pi^2\ga^4}  \right\}^{-\frac32}  
	=  \ga^2 81\sqrt{8} \frac{\pi^3}{g\beta(\frac\pi{2})^3}  \left(\frac1d+\frac12\right).  
	\]
	By \eqref{DEFineq} the other term in \eqref{Taylor} is estimated as 
	\[ 
	\frac{F}{E-D}  <  \frac{9\pi }{\beta(\frac\pi{2})}.  
	\]
	Therefore from \eqref{A bound3} we have 
	\[
	\sA  <   \frac{F}{E-D}  +  F^2 \Xi^{-\frac32}  
	<  \frac{9\pi}{\beta(\frac\pi{2})}   +  \ga^2 81\sqrt{8} \frac{\pi^3}{g\beta(\frac\pi{2})^3}  \left(\frac1d+\frac12\right) 
	<  \frac{12\pi }{\beta(\frac\pi{2})}
	\]
	provided \eqref{gamma ineq1} is satisfied.   
	Altogether, the smallness conditions on $\ga$ are \eqref{NmQ ineq} and \eqref{gamma ineq1}.  
\end{proof}

\subsection{Alternative proof}  
We now provide an alternative proof for Theorem~\ref{thm:adverse}, in which, rather than controlling the slope of the wave, we place bounds on its convexity.
We continue the estimates in Subsection 4.2 with modifications.  

Handling $\mathcal{J}_2$ differently from before, we make use of the behavior of only one of the three factors in the product \eqref{K1} at $s=0$. Since $f'(0)=f'(\pi)=0$, a Taylor expansion with remainder gives us 
\begin{equation}\label{K2N}
f(s)-f(0) = \frac12 f''(s^*) s^2 \le \frac12 M s^2 \quad \text{ for } s\in (0,\delta),
\end{equation} 
and similarly for $f(\pi)-f(\pi-s)$.  Here we used the fundamental assumption that
\begin{equation}
\sup_I |f''(s)| \le M \quad \text{where } I = (0,\delta) \cup (\pi-\delta,\pi),
\end{equation}   
where $M$ is a fixed large number and $\delta$ is a fixed number.  

As before, we break the interval of integration in \eqref{K1} into $(0,\delta)\cup(\pi-\delta,\pi)$. Using the estimate \eqref{K2N} on one of the factors, the $s^2$ singularity cancels, and we get  
\begin{equation}
\int_0^\delta\frac{\mathcal{I}(s)}{s^2} \,ds  \le  M\delta \ \{f(\pi)-f(0)\}^2.
\end{equation}  
Moreover, 
\begin{equation}
\int_\delta^\pi \frac{\mathcal{I}(s)}{s^2} \,ds \le  2\int_\delta^\pi \frac{ds}{s^2} \{f(\pi)-f(0)\}^3  < \frac1\delta \{f(\pi)-f(0)\}^3 . 
\end{equation} 
Combining, we thus obtain the estimate

\begin{equation}
\mathcal{K} f\Big|_0^\pi  <\frac{2}{3}\bigg(\frac{1}{d}+\frac{1}{2}+\frac{2}{\pi\delta}\bigg)\sA^3+  \frac{2M}{3\pi}\delta \sA^2 . 
\end{equation}  
In order to minimize this expression, we choose $\delta=\sqrt{\sA /M}$, which yields the bound 
\begin{equation}
\mathcal{K} f\Big|_0^\pi  <\frac{2}{3}\bigg(\frac{1}{d}+\frac{1}{2}\bigg)\sA^3  
+  \frac{8\sqrt{M}}{3\pi} \sA^{5/2} 
\end{equation} 
We are now ready to give the alternative proof  that involves $M$.  

\begin{theorem}
	Given a solution of $\mathcal{C}$, let the conformal depth $d>0$ and the convexity bound $M>0$ be fixed. 
	
(i)  Then for any $\ep>0$ there exists a small $\gamma_0>0$ such that for all $\gamma<\gamma_0$ we obtain the amplitude bounds
	\begin{equation}
	\mathscr A <\frac{2(1+\ep)\pi}{\beta(\pi/2)}\big(1+O(\gamma)\big)
	\end{equation}
	for some $\ep>0$, for water wave solutions on the curve $\mathcal{C}$. 
	
(ii) More specifically, given a solution of $\mathcal{C}$ with depth $d$, $Q$, and $\gamma$, we have 
	\begin{equation} 
	\sA <   \frac{2(1+\ep)\pi}{\beta(\frac\pi{2})}
	\left\{ 1 + \frac{|m|\gamma}{ gd} 
	+ \frac{\gamma^2 Q^2}{8dg^3} \right\} \quad \text{ provided } \quad P\bigg(\frac{F}{E}(1+\ep)\bigg)<0,  
	\end{equation}
	where the polynomial $P$ is given below in \eqref{PY} and the coefficients $F$ and $E$ are given in \eqref{bounds4N}.
\end{theorem}

\begin{proof}
For convenience we copy the key identity from \eqref{f_diff}, as follows:    
	\begin{equation}
	\left\{ 1-\frac{m\gamma}{dg}+\frac{\gamma^2 Q^2}{8dg^3}-\frac{\gamma^2[f^2]}{2dg}  \right\}  f\bigg|_{0}^\pi
	-\frac{\gamma^2}{2g}f^2\bigg|_{0}^\pi
	=  \{f\HH f'+\HH (ff') \}\bigg|_{0}^\pi - \frac{\gamma^2}{2g} \KK f \bigg|_{0}^\pi.
	\end{equation}
By  Lemmas~\ref{lem:bounds_CSV} and \ref{lem:bounds_K}, we get
	\begin{equation}
	\begin{split}
	\bigg(1-\frac{m\gamma}{dg}+\frac{\gamma^2 Q^2}{8dg^3}-\frac{\gamma^2[f^2]}{2dg}  \bigg)\sA  
	&\geq\bigg(\frac{\gamma^2}{2g}+\frac{\beta(\pi/2)}{2\pi}\bigg)\sA ^2
	-\frac{\gamma^2}{3g}\bigg(\frac{1}{d}+\frac{1}{2}\bigg)\sA ^3 
	-  \frac{4\gamma^2\sqrt{M}}{3\pi g} \sA ^{5/2}.
	\end{split} \sA 
	\end{equation}
Dividing by $\sA $ yields
	\begin{equation}\label{bounds3N}
	\begin{split}
	1-\frac{m\gamma}{dg}+\frac{\gamma^2 Q^2}{8dg^3}-\frac{\gamma^2[f^2]}{2dg}  
	\geq\bigg(\frac{\gamma^2}{2g}+\frac{\beta(\pi/2)}{2\pi}\bigg)\sA 
	-\frac{\gamma^2}{3g}\bigg(\frac{1}{d}+\frac{1}{2}\bigg)\sA ^2
	-  \frac{4\gamma^2\sqrt{M}}{3\pi g} \sA ^{3/2}.
	\end{split}
	\end{equation}
We  divide by the leading coefficient to obtain the inequality
	\begin{equation}\label{P}
	\sA ^2+D_M\sA ^{3/2}-E\sA +F\geq0,
	\end{equation}
which is not a quadratic, where 
	\begin{equation}\label{bounds4N}
	\begin{split}
	D_M&:=\frac{4\sqrt{M}}{\pi(\frac{1}{d}+\frac{1}{2})},\\
	E&:       
	= \frac{3}{2(\frac{1}{d}+\frac{1}{2})} + \frac{3g\beta(\frac\pi 2)} {2\pi\gamma^2 (\frac{1}{d}+\frac{1}{2})} ,\\ 
	F&:= \frac{3g}{\frac{1}{d}+\frac12} \left\{ \frac1{\gamma^2} - \frac{m}{\gamma gd} 
	+ \frac1{2dg}\left( \frac{Q^2}{4g^2}-[f^2]\right) \right\} .  
	\end{split}
	\end{equation}
Clearly, both $D_M$ and $E$ are positive. 
Note that both $E$ and $F$ have order $O(\gamma^{-2})$ as $\gamma\to 0$.  
Note also that $f\le \frac{Q}{2g}$, so that $ \frac{Q^2}{4g^2} > [f^2] >0$. 
{\it It is here that we assume the adverse case $\gamma>0$ and use the fact that $m<0$, so that $F$ is positive as well. } 
We set $Y=\sqrt{\sA}$ to obtain the polynomial inequality 
	\begin{equation}   \label{PY} 
	P(Y):=Y^4+D_MY^3-EY^2+F \ge 0.
	\end{equation}
	As this is a quartic, it does not have very simple roots.  
	
	With all these preparations, we have arrived at the most critical step in our reasoning.  
	It is clear that $P(0)=P'(0)=0$ and $P''(0)<0$, so that $Y=0$, which corresponds to laminar flow, 
	is a local maximum of $P$.  
	Thus $P$ has a unique local minimum at some $Y_{min}>0$.  
	Let us suppose that $P(Y_{min})<0$.  
	Then there would be a zero $Y_0$ of $P$ such that $P$ is positive in the finite interval $(0,Y_0)$.  
	In that case, any solution in the connected set $\mathcal C$ must satisfy $Y\le Y_0$.  
	Thus we would have the {\it amplitude bound} $\sA = \eta(0,0)-\eta(\pi,0) = f(\pi)-f(0) \le Y_0^2$.  
	{\it Hence $Y_1^2$ is also an amplitude bound for any $Y_1>0$ for which $P(Y_1)<0$.}  
	We proceed to look for such a $Y_1$.  
	
	Motivated by the last two terms in \eqref{PY}, we choose 
	$Y_1 = \sqrt{\frac{F}{E}(1+\ep)}$ for some $\ep>0$.  
	Then $P(Y_1)= Y_1^4+D_MY_!^3-\ep F$.  
	For sufficiently small $\gamma$, both $Y_1$ and $D_M$ are bounded while $F$ grows like $\gamma^{-2}$, 
	so that $P(Y_1)<0$, as desired.  
Using the $Y_1$ already chosen above, we have  
	\begin{equation} 
	\sA_1:=Y_1^2 = (1+\ep)\frac FE =  2(1+\ep) \left\{ \frac1{\gamma^2} - \frac{m}{\gamma gd} 
	+ \frac1{2dg}\left( \frac{Q^2}{4g^2}-[f^2]\right) \right\} 
	\left\{\frac1g + \frac{\beta(\frac\pi{2})} {\pi\gamma^2} \right\}^{-1}. 
	\end{equation}
Dropping $\frac1g$ and $[f^2]$, we obtain the upper bound on the wave amplitude  
	\begin{equation}  \label{X1ub1}
	\sA_1 <   \frac{2(1+\ep)\pi}{\beta(\frac\pi{2})}
	\left\{ 1 + \frac{|m|\gamma}{ gd} 
	+ \frac{\gamma^2 Q^2}{8dg^3} \right\} \quad \text{ provided } \quad P(X_1)<0.  
	\end{equation}
\end{proof} 

We now state explicit bounds for explicit small $\gamma$. These are just particular choices and are by no means necessary conditions.

\begin{corollary}
	Provided the bounds
	\begin{equation} \label{ineqgamma1}
	\gamma<\frac{gd}{|m|}, \quad \text{ and }\quad \gamma^2<\frac{8dg^3}{Q^2}, 
	\end{equation} 
	and 
	\begin{equation} \label{ineqgamma2}  
	\sqrt{M}\gamma^2 < \frac{\pi g}{4} \left( \frac{\beta(\frac\pi{2})} {12\pi} \right)^{\frac32}  \quad \text{ and } 
	\gamma^2 < \frac{2g}{\frac12 + \frac1d}  \left( \frac{\beta(\frac\pi{2})} {12\pi} \right)^2,
	\end{equation}
	hold, then the amplitude is bounded by
	\begin{equation}
	\mathscr A<\frac{12\pi}{\beta(\pi/2)}.
	\end{equation} 	
\end{corollary}

\begin{proof}
	From the choice in \eqref{ineqgamma1}, using \eqref{X1ub1}, we get the upper bound 
	\begin{equation}   \label{X1ub}
	\sA \le \sA_1 < \frac{12\pi}{\beta(\frac\pi{2})}.  
	\end{equation}
	In order to make sure that $P(X_1)<0$ in \eqref{X1ub1}, we make additional choices for how small $\gamma$ is. Recall that we have 
	\begin{equation}
	P(Y_1) = \sA_1^2 + D_M\sA_1^{3/2} - F. 
	\end{equation}
	Using the upper bound \eqref{X1ub} on $\sA_1$ and a lower bound on $F$, we obtain
	\begin{equation}
	P(Y_1) < \left( \frac{12\pi}{\beta(\frac\pi{2})} \right)^2  
	+  \frac{4\sqrt{M}}{\pi(\frac1d+\frac12)}  \left( \frac{12\pi}{\beta(\frac\pi{2})} \right)^{3/2}  
	- \frac{3g}{\gamma^2 (\frac1d+\frac12)}
	\end{equation}
	We want the last term to dominate the first two. 
	It is clear that {if $\gamma$ is small enough}, then $P(Y_1)<0$. The conditions given in \eqref{ineqgamma2} would be a pair of simpler explicit choices for the smallness of $\gamma$. 
\end{proof}

	Inequalities \eqref{ineqgamma1} and \eqref{ineqgamma2} are four simple conditions on the smallness 
	of $\gamma$ that guarantee an upper bound on the wave amplitude.  
	Summarizing the four conditions, they state that 
	$\gamma^2, \sqrt{M}\gamma^2, Q^2\gamma^2$ and $|m|\gamma^2$ 
	are smaller than some simple explicit expressions involving $g$ and $d$. We reiterate that they are by no means necessary conditions!

This completes the alternative proof of Theorem \ref{thm:adverse}.  


\section*{Acknowledgements}
The work of SVH is partially funded by the National Science Foundation through the award DMS-2102961.

\end{document}